\documentclass{amsart}

\usepackage{amsmath}
\usepackage{amsthm}
\usepackage{amsfonts}
\usepackage{amssymb}
\usepackage{graphics}
\usepackage{fullpage}

\newcommand\R{{\mathbb{R}}}
\newcommand\C{{\mathbb{C}}}

\newcommand\D{{\mathbf{D}}}

\renewcommand\P{{\mathbf{P}}}
\newcommand\E{{\mathbf{E}}}

\renewcommand\Im{{\operatorname{Im}}}
\renewcommand\Re{{\operatorname{Re}}}
\newcommand\eps{{\varepsilon}}

\newcommand\tr{\operatorname{trace}}

\renewcommand\Pr{{\mathbf P }}

\theoremstyle{plain}
  \newtheorem{theorem}{Theorem}[section]

  \newtheorem{proposition}[theorem]{Proposition}
  
  \newtheorem{lemma}[theorem]{Lemma}
  \newtheorem{corollary}[theorem]{Corollary}
   
\theoremstyle{remark}
  \newtheorem{remark}[theorem]{Remark}

\theoremstyle{definition}
  \newtheorem{definition}[theorem]{Definition}

\include{psfig}

\begin{document}

\title[Outliners for bounded rank perturbations]{Outliers in the spectrum of iid matrices with bounded rank perturbations}

\author{Terence Tao}
\address{Department of Mathematics, UCLA, Los Angeles CA 90095-1555}
\email{tao@math.ucla.edu}
\thanks{T. Tao is is supported by a grant from the MacArthur Foundation, by NSF grant DMS-0649473, and by the NSF Waterman award.}

\begin{abstract}  It is known that if one perturbs a large iid random matrix by a bounded rank error, then the majority of the eigenvalues will remain distributed according to the circular law.  However, the bounded rank perturbation may also create one or more outlier eigenvalues.  We show that if the perturbation is small, then the outlier eigenvalues are created next to the outlier eigenvalues of the bounded rank perturbation; but if the perturbation is large, then many more outliers can be created, and their law is governed by the zeroes of a random Laurent series with Gaussian coefficients.  On the other hand, these outliers may be eliminated by enforcing a row sum condition on the final matrix.
\end{abstract}

\maketitle

\setcounter{tocdepth}{2}
%\tableofcontents

\section{Introduction}

This paper is concerned with the study of outliers of the \emph{circular law} for iid random matrices and its variants.  To recall this law, we make some definitions:

\begin{definition}[iid random matrix]  An \emph{iid random matrix} is an $n \times n$ random matrix $X_n = (x_{ij})_{1 \leq i,j \leq n}$ (or more precisely, a nested sequence $X_1, X_2, \ldots$ of such matrices) whose entries $x_{ij}$ for $i,j \geq 1$ are independent identically distributed complex entries, which we normalise to have mean zero and variance one.  We say that such a matrix has \emph{atom distribution} $x$ if all the $x_{ij}$ have distribution $x$, thus $\E x = 0$ and $\E |x|^2 = 1$.
\end{definition}

\begin{definition}[ESD]  Given an $n \times n$ complex matrix $A_n$ (not necessarily Hermitian or normal), we define the \emph{empirical spectral distribution} $\mu_{A_n}$ of $A_n$ to be the probability measure
$$ \mu_{A_n} := \frac{1}{n} \sum_{j=1}^n \delta_{\lambda_j}$$
where $\lambda_j = \lambda_j(A_n)$ for $j=1,\ldots,n$ are the eigenvalues of $A_n$ (counting multiplicity, and ordered arbitrarily).

If $A_n$ is a random $n \times n$ complex matrix (so that $\mu_{A_n}$ is also random), we say that $\mu_{A_n}$ converges \emph{in probability} (resp. \emph{almost surely}) to another (Borel) probability measure $\mu$ on the complex plane $\C$ if for every smooth, compactly supported function $F: \C \to \C$, $\int_\C F\ d\mu_{A_n}$ converges in probability (resp. almost surely) to $\int_\C F\ d\mu$.
\end{definition}

The following theorem is the culmination of the work of many authors \cite{gin}, \cite{meh}, \cite{girko}, \cite{bai-first}, \cite{bai}, \cite{gotze}, \cite{twenty}, \cite{PZ}, \cite{TV-circular}, \cite{gt2}, \cite{tv-circular2}:

\begin{theorem}[Circular law for iid matrices]\label{thi}  Let $X_n$ be an iid random matrix.  Then $\mu_{\frac{1}{\sqrt{n}} X_n}$ converges almost surely (and hence also in probability) to the circular measure $\mu_c$, where $d\mu_c := \frac{1}{\pi} 1_{|z| \leq 1}\ dz$.
\end{theorem}

The result as stated is \cite[Theorem 1.15]{tv-circular2}, but this result is based on a large number of partial results (in which more hypotheses are placed on the atom distribution $x$) which are proven in the previously cited papers.

The circular law implies in particular that the spectral radius
$$ \rho( \frac{1}{\sqrt{n}} X_n ) = \lim_{k \to \infty} \left\| (\frac{1}{\sqrt{n}} X_n)^k \right\|_{op}^{1/k} = \sup_{1 \leq j \leq n} \left|\frac{1}{\sqrt{n}} \lambda_j(X_n)\right|$$
is at least $1-o(1)$ almost surely, where $o(1)$ goes to zero as $n \to \infty$.  When the atom distribution $x$ has finite fourth moment, we in fact have an asymptotic for the spectral radius:

\begin{theorem}[No outliers for iid matrices]\label{noout}  Let $X_n$ be an iid random matrix whose atom distribution $x$ has finite fourth moment: $\E |x|^4 < \infty$.  Then $\rho(\frac{1}{\sqrt{n}} X_n)$ converges to $1$ almost surely (and hence also in probability) as $n \to \infty$.  In fact, for any finite $m \geq 1$, $\| (\frac{1}{\sqrt{n}} X_n)^m \|_{op}$ converges to $m+1$ almost surely (and hence also in probability) as $n \to \infty$.

Furthermore, if all moments of $x$ are finite, one has the tail bound
$$ \left\| (\frac{1}{\sqrt{n}} X_n)^m \right\|_{op} \leq m+1+\eps $$
with overwhelming probability\footnote{We say that an event $E_n$ depending on $n$ occurs with \emph{overwhelming probability} if for every $A > 0$ there exists $C>0$ such that $\P( E_n ) \geq 1 - C n^{-A}$ for all $n$.} for each $\eps > 0$.
\end{theorem}

\begin{proof} This follows from what is by now a routine application of the truncation method and the moment method; see \cite{baiyin}, \cite{geman}, or \cite[Theorem 5.17]{BY}. We remark that the tail bound can also be deduced from the main result using the Talagrand concentration inequality (after first truncating to the case when $x$ is bounded); see \cite{gz}, \cite{akv}, \cite{meckes}.  The precise expression $m+1$ is not important for our arguments here; any quantity that was subexponential in $m$ would have sufficed.
\end{proof}

Informally, Theorem \ref{noout} asserts that when the fourth moment is finite, there are no significant \emph{outliers} to the circular law: with probability\footnote{The asymptotic notation $O(), o()$ that we use here will be defined in Section \ref{notation-sec}.} $1-o(1)$, all of the eigenvalues of the matrix $\frac{1}{\sqrt{n}} X_n$ lie within $o(1)$ of the support $\{ z \in \C: |z| \leq 1 \}$ of the circular law.  The fourth moment condition here is necessary for the second conclusion of Theorem \ref{noout} (see \cite{BY}), and it is very likely that it is also necessary for the first conclusion.

Now we consider the circular law and its outliers for random matrices formed as a low-rank perturbation of an iid random matrix.  The circular law is stable under such perturbations:

\begin{theorem}[Circular law for low rank perturbations of iid matrices]\label{lrc}\cite[Corollary 1.17]{tv-circular2} Let $X_n$ be an iid random matrix, and for each $n$, let $C_n$ be a deterministic matrix with rank $o(n)$ obeying the Frobenius norm bound 
$$ \|C_n\|_F := (\tr C_n C_n^*)^{1/2} = O(n^{1/2}).$$
Then $\mu_{\frac{1}{\sqrt{n}} X_n + C_n}$ converges both in probability and in the almost sure sense to the circular measure $\mu_c$.
\end{theorem}

\begin{remark} Thanks to a recent result of Bordenave \cite{bordenave}, the $O(n^{1/2})$ bound here can be relaxed to $O(n^{O(1)})$.
\end{remark}

However, the low rank perturbation $C_n$ can now create outliers.  Our first main result is to describe these outliers in the case when $C_n$ has bounded rank and bounded operator norm, and $x$ has finite fourth moment.  In this case, it turns out that the outliers of $\frac{1}{\sqrt{n}} X_n + C_n$ are close to those of $C_n$.  More precisely, we have

\begin{theorem}[Outliers for small low rank perturbations of iid matrices]\label{small-rank} Let $X_n$ be an iid random matrix whose atom distribtuion has finite fourth moment, and for each $n$, let $C_n$ be a deterministic matrix with rank $O(1)$ and operator norm $O(1)$.  Let $\eps > 0$, and suppose that for all sufficiently large $n$, there are no eigenvalues of $C_n$ in the band $\{ z \in \C: 1+\eps < |z| < 1+3\eps \}$, and there are $j$ eigenvalues $\lambda_1(C_n),\ldots,\lambda_j(C_n)$ for some $j=O(1)$ in the region $\{z \in \C: |z| \geq 1+3\eps \}$.  Then, almost surely, for sufficiently large $n$, there are precisely $j$ eigenvalues $\lambda_1(\frac{1}{\sqrt{n}} X_n + C_n), \ldots, \lambda_j(\frac{1}{\sqrt{n}} X_n + C_n)$ of $\frac{1}{\sqrt{n}} X_n + C_n$ in the region $\{ z \in \C: |z| \geq 1+2\eps \}$, and after labeling these eigenvalues properly, $\lambda_i(\frac{1}{\sqrt{n}} X_n + C_n) = \lambda_i(C_n) + o(1)$ as $n \to \infty$ for each $1 \leq i \leq j$.
\end{theorem}

Thus, for instance, if one perturbs $\frac{1}{\sqrt{n}} X_n$ by a bounded rank, bounded operator norm matrix $C_n$ whose eigenvalues all lie inside the unit disk $\D := \{ z \in \C: |z| \leq 1 \}$ (e.g. $C_n$ could be a nilpotent matrix), then no outliers are created; but once $C_n$ has eigenvalues leaving the unit disk, the perturbed matrix $\frac{1}{\sqrt{n}} X_n + C_n$ will also have outliers in asymptotically the same location.  Theorem \ref{small-rank} is illustrated in Figures \ref{Thm1.6Fig}, \ref{Thm1.6Figb}.

\begin{figure}
\begin{center}
\scalebox{.7}{\includegraphics{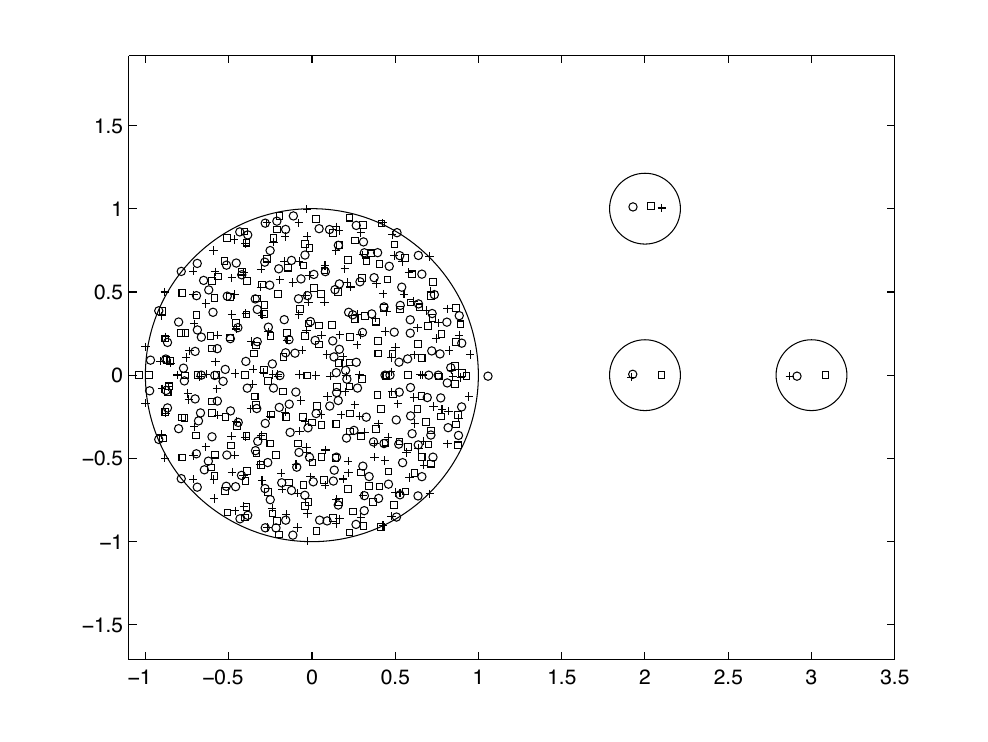}}
\end{center}
\caption{This figure shows the eigenvalues of three $200$ by $200$ iid random matrices with atom
distribution $x$ defined by $\Pr(x=1) = \Pr(x=-1) = 1/2$, each of which was
perturbed by adding the matrix $\operatorname{diag}(2+i,3,2,0,0,\dots,0)$.
The small circles are centered at $2+i$, $2$, and $3$, respectively, and each
have radius $n^{-1/4}$ where $n=200$.  (Figure by Phillip Wood.)}
\label{Thm1.6Fig}
\end{figure}

\begin{figure}
\begin{center}
\scalebox{.6}{\includegraphics{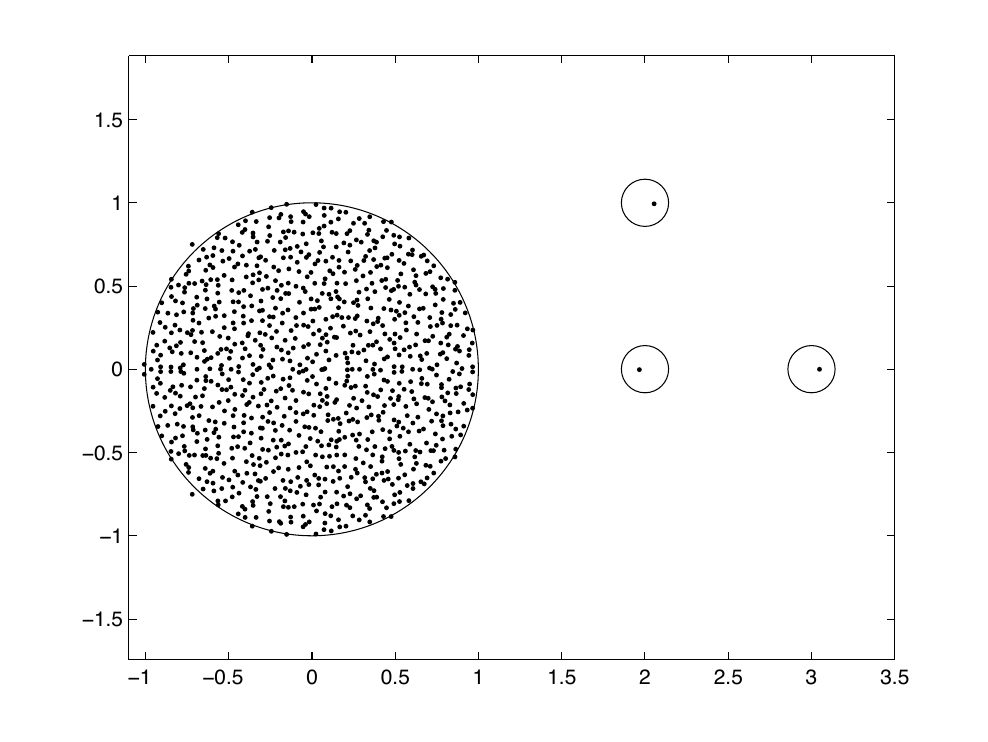}}
\end{center}
\caption{This figure shows the eigenvalues of a single $1000$ by $1000$ iid
random matrix with atom distribution $x$ defined by $\Pr(x=1) = \Pr(x=-1) =
1/2$ which was perturbed by adding the matrix
$\operatorname{diag}(2+i,3,2,0,0,\dots,0)$.  The small circles are centered at
$2+i$, $2$, and $3$, respectively, and each have radius $n^{-1/4}$ where
$n=1000$. (Figure by Phillip Wood.)}
\label{Thm1.6Figb}
\end{figure}

\begin{remark} An analogous result for Wigner matrices instead of iid matrices has recently been established in \cite{cdf}, \cite{prs}, with more precise control (in particular, a central limit theorem) on the distribution of the outlier eigenvalues; the methods used are somewhat different, but the techniques developed here can be adapted to the Wigner case (Alexander Soshnikov, private communication).  See also \cite{feral}, \cite{br} for further results in the Wigner case, whose methods are close to those used here, \cite{peche} for a treatment of the GUE case, \cite{bby} for a treatment of the LUE case, and \cite{bai-exact}, \cite{bs}, \cite{baik} for a treatment of the covariance matrix case.  Interestingly, in the Wigner case the outlier eigenvalues $\lambda_i(\frac{1}{\sqrt{n}} X_n + C_n)$ of the perturbed matrix are not close to the outlier eigenvalues $\lambda_i(C_n)$ of the original matrix, but rather to the shifted eigenvalues $\lambda_i(C_n) + \frac{\sigma^2}{\lambda_i(C_n)}$, where $\sigma^2$ is the variance of the entries of the Wigner matrix $\sigma$.  This is ultimately because the powers $(\frac{1}{\sqrt{n}} X_n)^m$ have a significant presence on the diagonal in the Wigner case, in contrast with the iid case where all entries are small.  Alternatively: the Wigner semicircular law has nonzero moments, while all nontrivial (pure) moments of the circular law vanish.
\end{remark}

Theorem \ref{small-rank} is proven in Section \ref{slrp}.  The main tools are asymptotics of Stieltjes transforms outside of the unit disk $\D$, combined with the fundamental matrix identity\footnote{We thank Percy Deift for emphasising the importance of this identity in random matrix theory.}
\begin{equation}\label{det-ident}
\det(1+AB) = \det(1+BA)
\end{equation}
valid for arbitrary $n \times k$ matrices $A$ and $k \times n$ matrices $B$.  Note that the left-hand side is an $n \times n$ determinant, while the right-hand side is a $k \times k$ determinant.  For low rank perturbations, we will be able to apply \eqref{det-ident} with $k$ bounded and $n$ going to infinity, allowing one to transform an unbounded-dimensional problem into a finite-dimensional one.

Theorem \ref{small-rank} only deals with perturbations that are relatively small, having an operator norm of $O(1)$.  It is also of interest to consider larger perturbations, such as those caused by adjusting the mean of each coefficient of $X_n$ by $O(1)$.  Here, the situation is more complicated, and we will consider only a few model perturbations, rather than attempt to obtain the most general result.

We first consider the case of iid matrices with non-zero mean, which we write as $\frac{1}{\sqrt{n}} X_n + \mu \sqrt{n} \phi_n \phi_n^*$, where $\mu$ is a fixed complex number (independent of $n$) and $\phi_n$ is the unit column vector $\phi_n := \frac{1}{\sqrt{n}} (1,\ldots,1)^*$; this corresponds to shifting the atom distribution $x$ by $\mu$ (so that it has mean $\mu$ rather than mean zero).  This is a rank one perturbation of $\frac{1}{\sqrt{n}} X_n$.  The circular law still holds for this ensemble, thanks to Theorem \ref{lrc} (or the earlier result of Ch\"afai \cite{cha}).  However, in view of Theorem \ref{small-rank}, we expect a single large outlier near $\mu \sqrt{n}$.  This is indeed the case:

\begin{theorem}[Outlier for iid matrices with nonzero mean]\label{nonzero-mean} Let $X_n$ be an iid random matrix whose atom distribution has finite fourth moment, and let $\mu \in \C$ be a non-zero quantity independent of $n$.  Then almost surely, for sufficiently large $n$, all the eigenvalues of $\frac{1}{\sqrt{n}} X_n + \mu \sqrt{n} \phi_n \phi_n^*$ lie in the disk $\{ z \in \C: |z| \leq 1+o(1) \}$, with a single exception taking the value $\mu \sqrt{n} + o(1)$.
\end{theorem}

We prove this result in Section \ref{nzm-sec}.  One can obtain more precise information on the distribution of this exceptional eigenvalue, particularly if one assumes more moment hypotheses on the atom distribution $x$; see \cite{silv}.  The existence of this exceptional eigenvalue was already noted back in \cite{andrew}.  Theorem \ref{nonzero-mean} is illustrated in Figure \ref{Thm1.8Fig}.  Figure \ref{Thm1.8Figb} corresponds to the case of a smaller value of $\mu$, and falls instead under the regime covered by Theorem \ref{small-rank}.

\begin{figure}
\begin{center}
\scalebox{.7}{\includegraphics{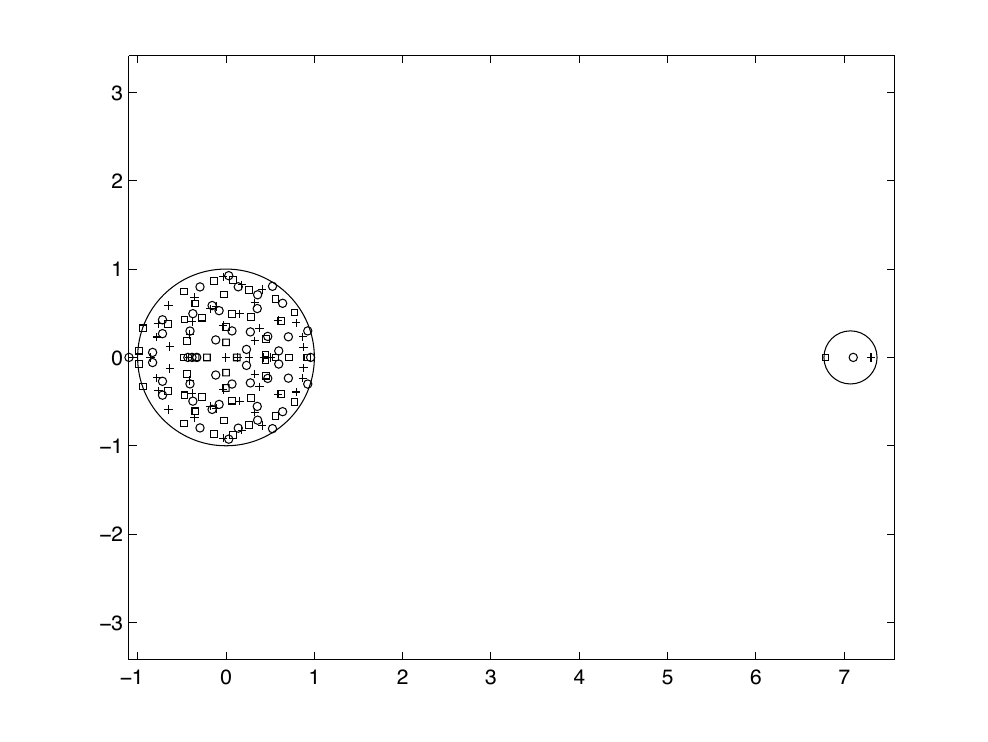}}
\end{center}
\caption{This figure shows the eigenvalues of three $50$ by $50$ iid random matrices with atom
distribution $x$ defined by $\Pr(x=1) = \Pr(x=-1) = 1/2$, each of which was
perturbed by adding the matrix $\mu \sqrt n \phi_n \phi_n^*$ where $\mu=1$.
The small circle is centered at $\sqrt{50}$ and has radius $n^{-1/4}$ where
$n=50$.  (Figure by Phillip Wood.)}
\label{Thm1.8Fig}
\end{figure}

\begin{figure}
\begin{center}
\scalebox{.6}{\includegraphics{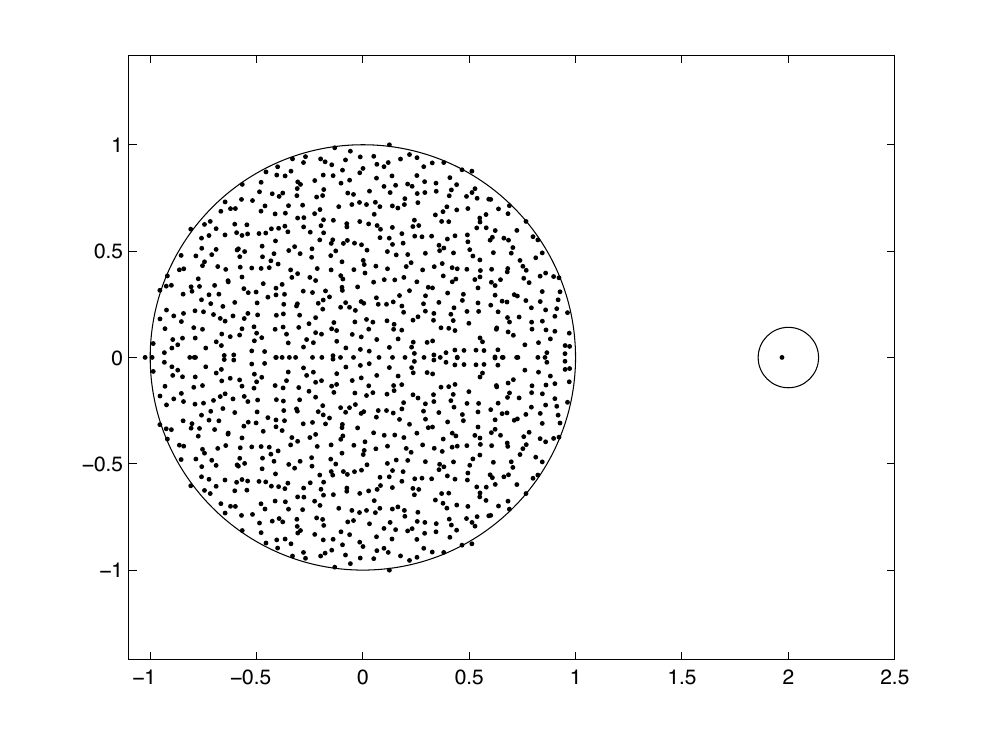}}
\end{center}
\caption{This figure shows the eigenvalues of a single $1000$ by $1000$ iid
random matrix with atom
distribution $x$ defined by $\Pr(x=1) = \Pr(x=-1) = 1/2$ which was
perturbed by adding the matrix $\mu \sqrt n \phi_n \phi_n^*$ where
$\mu=2/\sqrt{1000}$.
The small circle is centered at $2$ and has radius $n^{-1/4}$ where $n=1000$. (Figure by Phillip Wood.)}
\label{Thm1.8Figb}
\end{figure}

Next, we consider a model that was introduced in \cite{rajan}, in the context of neural networks.  In our notation, this model takes the form
\begin{equation}\label{ann}
 A_n := \frac{1}{\sqrt{n}} X_n + \mu B_n
\end{equation}
where $\mu > 0$ is a fixed parameter, and $B_n$ is a random matrix (independent of $X_n$) such that the columns of $B_n$ are iid, with each column equal to $\sqrt{\frac{1-p}{p}} \phi_n$ with probability $p$ and $-\sqrt{\frac{p}{1-p}} \phi_n$ with probability $1-p$, for some fixed $0 < p < 1$.  (In the notation of \cite{rajan}, the excitatory mean $\mu_E$ is $\mu \sqrt{(1-p)/p}$ and the inhibitory mean $\mu_I$ is $-\mu\sqrt{p/(1-p)}$, and the excitatory and inhibitory variances are assumed to be equal.)  Note that one can write
\begin{equation}\label{ann0}
A_n = \frac{1}{\sqrt{n}} X_n + \mu \sqrt{n} \phi_n \psi_n^*
\end{equation}
where $\psi_n$ is a random vector whose entries are iid and equal $\frac{1}{\sqrt{n}} \sqrt{\frac{1-p}{p}}$ with probability $p$ and $-\frac{1}{\sqrt{n}} \sqrt{\frac{p}{1-p}}$ with probability $1-p$; with this normalisation, $\psi_n$ has mean zero and unit variance.

Again, by Theorem \ref{lrc}, the ESD of $A_n$ is governed by the circular distribution $\mu_c$ in the limit $n \to \infty$; however, as observed numerically in \cite{rajan}, a small number of outliers also appear for $A_n$.  

It is possible to explain the outliers by the arguments of this paper.  However, in contrast to the situations in Theorem \ref{small-rank} or Theorem \ref{nonzero-mean}, in which the outliers essentially have a deterministic location up to $o(1)$ errors, for the model \eqref{ann}, the outliers retain significant randomness at macroscopic scales, and need to be modeled by a point process rather than by a deterministic law.  To describe this point process, we introduce the \emph{$k$-point correlation functions} $\rho^{(k)}_{A_n}: \C^k \to \R^+$ for the ESD of $A_n$ for $1 \leq k \leq n$, defined as the unique symmetric function such that
$$ \int_{\C^k} F(z_1,\ldots,z_k) \rho^{(k)}_{A_n}(z_1,\ldots,z_k)\ d^2 z_1 \ldots d^2 z_k = \frac{1}{k!} \E \sum_{i_1,\ldots,i_k \in \{1,\ldots,n\}, \hbox{ distinct}} F(\lambda_{i_1}(A_n), \ldots, \lambda_{i_k}(A_n))$$
for all continuous, compactly supported test functions $F: \C^k \to \C$; note that the right-hand side is not dependent on how one orders the $n$ eigenvalues of $A_n$.  Here, $d^2 z$ denotes two-dimensional Lebesgue measure on $\C$.  If the atom distribution of $A_n$ is discrete, then $\rho^{(k)}_{A_n}$ needs to be interpreted as a distribution or measure rather than as a function, but this technicality will not concern us here.

It turns out that the correlation functions $\rho^{(k)}_{A_n}$ have a limiting law $\rho^{(k)}_\infty$ outside of the unit disk $\D$ (inside the disk, one expects these functions to go to infinity, thanks to the circular law and the choice of normalisation).  We do not have a completely explicit formula for this limit, but can describe it instead as the zeroes of a random Laurent series.  More precisely, consider the random Laurent series
$$ g(z) := 1 - \mu \sum_{j=1}^\infty \frac{g_j}{z^j}$$
where $g_1, g_2, \ldots$ are iid copies of the real Gaussian distribution $N(0,1)_\R$.  From the Borel-Cantelli lemma we see that this Laurent series is almost surely convergent in the complement $\C \backslash \D$ of the closed unit disk $\D$, and almost surely has a finite number of zeroes in the region $\{ z \in \C: |z| > 1+\eps\}$ for any fixed $\eps > 0$.  We then define the limiting correlation function $\rho^{(k)}_\infty: (\C \backslash \D)^k \to \R^+$ outside of this disk as the unique symmetric function (or more precisely, distribution) such that
\begin{equation}\label{ck} \int_{\C^k} F(z_1,\ldots,z_k) \rho^{(k)}_{\infty}(z_1,\ldots,z_k)\ d^2z_1 \ldots d^2z_k = \frac{1}{k!} \E \sum_{w_1,\ldots,w_k \in \Lambda_g, \hbox{ distinct}} F(w_1,\ldots,w_k)
\end{equation}
for all continuous, compactly supported test functions $F: (\C \backslash \D)^k \to \C$, where $\Lambda_g$ are the zeroes of $g$ (counting multiplicity); a little more explicitly, $\rho^{(k)}_\infty(z_1,\ldots,z_k)$ can be defined for distinct $z_1,\ldots,z_k \in \C \backslash \D$ as
$$ \rho^{(k)}_\infty(z_1,\ldots,z_k) = \lim_{\eps \to 0} \frac{\rho^{(k),\eps}_\infty(z_1,\ldots,z_k)}{(\pi\eps^2)^k}$$
where $\rho^{(k),\eps}_\infty(z_1,\ldots,z_k)$ is the probability of the event that there is a zero of $g$ within $\eps$ of $z_j$ for each $j=1,\ldots,k$.  

\begin{remark}  If the $g_1,g_2,\ldots$ were \emph{complex} Gaussian $N(0,1)_\C$ instead of real, we normalised $\mu=1$, and we replaced the constant coefficient $1$ by another complex gaussian $g_0$, then $g(z)$ would be a Gaussian power series (GPS) in the variable $1/z$.  Gaussian power series have been intensively studied (see the recent text \cite{hough} and the references therein).  In that case, the $k$-point correlation functions are given by the determinantal formula
$$ \rho^{(k)}_{GPS}(z_1,\ldots,z_k) = \frac{1}{\pi^k} \det( \frac{1}{1 - z_i \overline{z_j}} )_{1 \leq i,j \leq k};$$
see \cite{peres}. One may then hope that a somewhat analogous formula might be obtained for the random Laurent series considered here, possibly using the explicit formulae for the correlation functions of zeroes of real random polynomials from \cite{prosen} as a starting point.  We will not pursue this matter.
\end{remark}

We can now state our main theorem regarding this model, which we prove in Section \ref{ull-sec}:

\begin{theorem}[Limiting law]\label{ull} Let $X_n$ be an iid random matrix whose atom distribution $x$ is real-valued and which is either gaussian (i.e. $x \equiv N(0,1)_\R$) or bounded, and let $0 < p < 1$ and $\mu > 0$ be fixed.  Let $A_n$, $\rho^{(k)}_{A_n}$, and $\rho^{(k)}_\infty$ be defined as above.  
\begin{enumerate}
\item[(i)] (Crude upper bound) For any $\eps > 0$, let $N_\eps$ denote the number of eigenvalues of $A_n$ in the region $\{ z \in \Omega: |z| \geq 1+\eps\}$ (counting multiplicity).  Then $\sup_n \E N_\eps^m < \infty$ for all $\eps > 0$ and $m \geq 1$.
\item[(ii)] (Limiting law) $\rho^{(k)}_{A_n}$ converges in the vague topology to $\rho^{(k)}_\infty$ on $(\C \backslash \D)^k$.  In other words, one has
$$ \int_{\C^k} F(z_1,\ldots,z_k) \rho^{(k)}_{A_n}(z_1,\ldots,z_k)\ d^2 z_1 \ldots d^2 z_k \to \int_{\C^k} F(z_1,\ldots,z_k) \rho^{(k)}_{\infty}(z_1,\ldots,z_k)\ d^2 z_1 \ldots d^2 z_k$$
whenever $F: (\C \backslash \D)^k \to \C$ is continuous and compactly supported (in particular, it is supported in the region $\{ (z_1,\ldots,z_k) \in \C^k: |z_1|,\ldots,|z_k| \geq 1+\eps\}$ for some $\eps>0$).  In particular, the limiting distribution is universal with respect to the distribution $x$.
\end{enumerate}
\end{theorem}

\begin{remark}  The requirement that all the coefficients of $X_n$ are real is a natural one from the neural net application\cite{rajan}.  However, in view of the better developed theory for complex Gaussian power series\cite{hough}, it may in fact be more natural from a theoretical perspective to consider the case when the $X_n$ are complex valued, e.g. if the atom distribution is complex Gaussian.  In that case, there is a similar result to Theorem \ref{ull} but with the coefficients $g_j$ of the random Laurent series $g(z)$ given by complex Gaussians rather than real Gaussians; we omit the details.  The requirement that $x$ be either Gaussian or bounded is a technical one, so that one may apply concentration inequalities; it may certainly be relaxed substantially.
\end{remark}

Theorem \ref{ull} is illustrated in Figure \ref{Thm1.10Fig}.

\begin{figure}
\begin{center}
\scalebox{.6}{\includegraphics{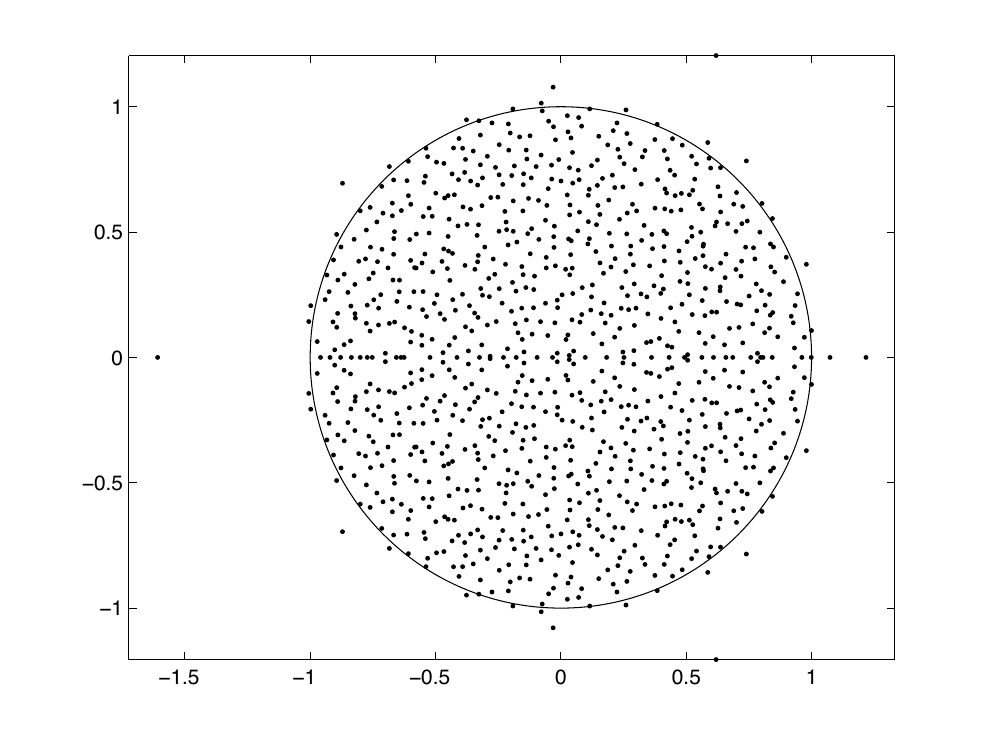}}
\end{center}
\caption{This figure shows the eigenvalues of a single $1000$ by $1000$ iid
random matrix with atom
distribution $x$ defined by $\Pr(x=1) = \Pr(x=-1) = 1/2$ which was
perturbed by adding the random matrix $\mu B_n$ as defined after \eqref{ann},
where $\mu=2$ and $p=1/4$.  (Figure by Phillip Wood.)}
\label{Thm1.10Fig}
\end{figure}

\subsection{Zero-sum matrices}  

Next, we consider a different low-rank perturbation of an iid matrix model, in which the row sums of the matrix are forced to equal zero.  Introduce the orthogonal projection matrix
$$ P_n = 1 - \phi_n \phi_n^* = ( \delta_{ij} - \frac{1}{n} )_{1 \leq i,j \leq n};$$
thus $P_n$ is the orthogonal projection onto the hyperplane $\phi_n^\perp = \{ (x_1,\ldots,x_n) \in \C^n: x_1 + \ldots + x_n = 0 \}$.

Our first result is that the presence of this projection does not affect the circular law, nor does it create any outliers:

\begin{theorem}\label{xp}  
Let $X_n$ be an iid random matrix.  Then $\mu_{\frac{1}{\sqrt{n}} X_n P_n}$ converges both in probability and almost surely to the circular measure $\mu_c$, and almost surely one has $\rho( \frac{1}{\sqrt{n}} X_n P_n ) = 1 + o(1)$.
\end{theorem}

We prove this theorem in Section \ref{zrc}, using the machinery from \cite{tv-circular2}.  The main difficulty is to ensure that the least singular value of $\frac{1}{\sqrt{n}} X_n P_n - z$ is well controlled for any fixed $z$, but this can be achieved dropping one dimension (and freezing some of the entries of $X_n$) to eliminate the role of $P_n$ (at the cost of replacing the deterministic matrix $z = zI$ by a more complicated matrix).  This result is somewhat analogous to the result of \cite{chafai} establishing the circular law for random Markov matrices (under an additional bounded density hypothesis on the atom distribution), although the situation is more complicated in that case because a slightly nonlinear transformation is required in order to convert a iid random matrix to a Markov matrix, in contrast to the simple linear transformation $X_n \mapsto X_n P_n$ required to make the rows sum to zero.

Suppose that $\psi_n$ is a vector orthogonal to $\phi_n$, then $\psi_n^* = \psi_n^* P_n$ and $P_n \phi_n = 0$.  Then from two applications of \eqref{det-ident} we have
\begin{align*}
\det\left( 1 + z (\frac{1}{\sqrt{n}} X_n P_n + \phi_n \psi_n^*) \right) &= \det\left( 1 + z P_n (\frac{1}{\sqrt{n}} X_n P_n + \phi_n \psi_n^*) \right) \\
&= \det\left( 1 + z P_n \frac{1}{\sqrt{n}} X_n P_n \right) \\
&= \det\left( 1 + z \frac{1}{\sqrt{n}} X_n P_n \right).
\end{align*}
We thus see that $\frac{1}{\sqrt{n}} X_n P_n$ and $\frac{1}{\sqrt{n}} X_n P_n + \phi_n \psi_n^*$ have the same characteristic polynomial, and thus the same ESD and the same spectral radius.  (One can also establish these facts directly without much difficulty, as was done in \cite{rajan}.)  We thus obtain

\begin{corollary}  
Let $X_n$ be an iid random matrix, and for each $n$, let $\psi_n$ be a (possibly random, and $X_n$-dependent) vector which is orthogonal to $\phi_n$.  Then $\mu_{\frac{1}{\sqrt{n}} X_n P_n + \phi_n \psi_n^*}$ converges both in probability and almost surely to the circular measure $\mu_c$, and almost surely one has $\rho( \frac{1}{\sqrt{n}} X_n P_n + \phi_n \psi_n^* ) = 1 + o(1)$.
\end{corollary}

In particular, no outliers are created no matter how large $\psi_n$ is (or how aligned it is with $\phi_n$).  These results were established in \cite{rajan} in the gaussian case.  This is in sharp contrast to the situation in Theorem \ref{ull} for the model \eqref{ann0}, which is similar to the matrix $\frac{1}{\sqrt{n}} X_n P_n + \phi_n \psi_n^*$, but without the $P_n$ projection.

\begin{remark}  Our results here are only effective in the region where the spectral parameter $z$ has magnitude larger than that of the spectral radius.  It would be of interest to determine what happens in models where the eigenvalue law of the base matrix $\frac{1}{\sqrt{n}} X_n$ is not governed by a circular law, but by another law whose support does not occupy the entire disk given by the spectral radius (e.g. matrices whose ESD is concentrated in an annulus).  In the covariance matrix case, results in this direction appear in \cite{bai-exact}, \cite{baik}.
\end{remark}

\subsection{Acknowledgments}

This work was conducted during a workshop at the American Institute of Mathematics.  We thank Larry Abbott for raising these questions.  After posing this question at the workshop, Alice Guionnet provided the key insight, namely to reduce matters to studying coefficients of the resolvent of $\frac{1}{\sqrt{n}} X_n$, while Percy Deift emphasised the significance of the identity \eqref{det-ident} to questions of this type (and indeed, this identity is crucial in order to efficiently handle the higher rank case $k > 1$).  The author also thanks Sasha Soshnikov and Phillip Wood for useful discussions, Florent Benaych-Georges, Djalil Chafai and Raj Rao for references, and Phillip Wood for corrections.  We are also indebted to Phillip Wood for supplying the figures for this paper.  Finally, we thank the anonymous referees for many helpful comments, corrections, and references.

\subsection{Asymptotic notation}\label{notation-sec}

Throughout this paper, $n$ is an asymptotic parameter going to infinity.  We use $o(1)$ to denote any quantity that is bounded in magnitude by an expression $c(n)$ that goes to zero as $n \to \infty$, keeping all other parameters independent of $n$ (e.g. $x, \eps, k, z$) fixed.  Similarly, we use $X=O(Y)$ or $X \ll Y$ to denote the estimate $|X| \leq CY$ where the implied constant $C$ is independent of $n$ but may depend on on parameters independent of $n$.

\section{Small low rank perturbation}\label{slrp}

We now prove Theorem \ref{small-rank}.  Fix $\eps > 0$ and $x$, $C_n$ as in that theorem.  By hypothesis, $C_n$ has rank at most $k$ for some $k=O(1)$ independent of $n$, and an operator norm of $O(1)$.  By the singular value decomposition, we can write $C_n = A_n B_n$ for some $n \times k$ and $k \times n$ matrices $A_n, B_n$, both of operator norm $O(1)$.

We have the following description\footnote{We are indebted to Alice Guionnet for proposing the $k=1$ case of this formula, as well as the basic strategy of proof used in this paper, and Percy Deift for emphasising the importance of the identity \eqref{det-ident}.} of the eigenvalues of $\frac{1}{\sqrt{n}} X_n + C_n$ in terms of a $k \times k$ determinant:

\begin{lemma}[Eigenvalue criterion]\label{eigencrit}  Let $z$ be a complex number that is not an eigenvalue of $\frac{1}{\sqrt{n}} X_n$.  Then $z$ is an eigenvalue of   
$\frac{1}{\sqrt{n}} X_n + C_n$ if and only if
\begin{equation}\label{dezn}
\det( 1 + B_n (\frac{1}{\sqrt{n}} X_n - z)^{-1} A_n ) = 0.
\end{equation}
\end{lemma}

\begin{proof}  Clearly, $z$ is an eigenvalue of $\frac{1}{\sqrt{n}} X_n + C_n$ if and only if
$$ \det( \frac{1}{\sqrt{n}} X_n + C_n - z ) = 0.$$
By hypothesis, $\frac{1}{\sqrt{n}} X_n - z$ is invertible and $C_n = A_n B_n$, so we may rewrite this equation as
$$ \det( 1 + (\frac{1}{\sqrt{n}} X_n - z)^{-1} A_n B_n ) = 0.$$
The claim now follows from \eqref{det-ident}.
\end{proof}

\begin{remark}\label{explicit} The above argument in fact shows that
\begin{align*}
\det( 1 + B_n (\frac{1}{\sqrt{n}} X_n - z)^{-1} A_n ) &= \frac{ \det( \frac{1}{\sqrt{n}} X_n + C_n - z ) }{\det( \frac{1}{\sqrt{n}} X_n - z) } \\
&= \frac{ \prod_{j=1}^n (\lambda_j(\frac{1}{\sqrt{n}} X_n + C_n) - z) }{ \prod_{j=1}^n (\lambda_j(\frac{1}{\sqrt{n}} X_n ) - z) }
\end{align*}
whenever the denominator is nonzero.  We are indebted to Alice Guionnet for suggesting the use of this type of criterion, versions of which also appear in \cite{agg}, \cite{baiyao}, \cite{cdf}, \cite{br}, \cite{bgm}, \cite{bgm2}.  For instance, the $k=1$ case of this criterion appears explicitly in \cite{br}.  In \cite{agg} the expression in this identity (which is stated there in the case of symmetric matrices) is referred to as the \emph{modified Weinstein determinant}.  (We thank Raj Rao for this reference.)
\end{remark}

Introduce the functions
\begin{align*}
f(z) &:= \det( 1 + B_n (\frac{1}{\sqrt{n}} X_n - z)^{-1} A_n ) \\
g(z) &:= \det( 1 + B_n (- z)^{-1} A_n ).
\end{align*}
These are both meromorphic functions that are asymptotically equal to $1$ at infinity, with $g$ being a rational function of degree at most $k$ with bounded coefficients.  Lemma \ref{eigencrit} tells us that outside of the spectrum of $\frac{1}{\sqrt{n}} X_n$, the zeroes of $f(z)$ agree with the eigenvalues of $\frac{1}{\sqrt{n}} X_n + C_n$.  An inspection of the argument also reveals that the multiplicity of a given such eigenvalue is equal to the degree of the corresponding zero of $f$.  Similarly, replacing $\frac{1}{\sqrt{n}} X_n$ by the zero matrix in Lemma \ref{eigencrit}, we see that outside of the origin, the zeroes of $g$ are precisely the eigenvalues of $C_n$ (counting multiplicity).  Indeed, from \eqref{det-ident} one has
$$ g(z) = \prod_{i=1}^k \left(1 - \frac{\lambda_i(C_n)}{z}\right),$$
where $\lambda_1(C_n),\ldots,\lambda_k(C_n)$ are the $k$ non-trivial eigenvalues of $C_n$ (some of which may be zero), including of course the 
$j$ eigenvalues $\frac{1}{\sqrt{n}} X_n$ of magnitude at least $1+3\eps$.

By Theorem \ref{noout}, we see that almost surely, the spectrum of $\frac{1}{\sqrt{n}} X_n$ is contained in the disk $\{ z \in \C: |z| \leq 1+\eps \}$ for sufficiently large $n$.  In view of Rouche's theorem (or the argument principle), together with the fact that the coefficients of $g$ are bounded, it then suffices to show that the quantity
$$ \sup_{|z| \geq 1+2\eps} |f(z)-g(z)|$$
converges almost surely to zero.  Since $k$ is fixed and $B_n, A_n$ are bounded in operator norm, it suffices to show that
$$ \sup_{|z| \geq 1+2\eps} \| B_n ((\frac{1}{\sqrt{n}} X_n - z)^{-1} - (-z)^{-1}) A_n \|_{op}$$
converges almost surely to zero.

By Theorem \ref{noout}, we almost surely have
$$ \left\| (\frac{1}{\sqrt{n}} X_n)^m \right\|_{op} = m+1 + o(1)$$
for each $m \geq 1$.  In particular, there exists $m_0 > 0$ such that
$$ \left\| (\frac{1}{\sqrt{n}} X_n)^{m_0} \right\|_{op} \leq (1+\eps)^{m_0}$$
for sufficiently large $n$, which also implies that one has
$$ \left\| (\frac{1}{\sqrt{n}} X_n)^{m} \right\|_{op} \leq K (1+\eps)^{m}$$
for all $n, m \geq 1$ and some almost surely finite random variable $K$.  This ensures that the Neumann series
$$ (\frac{1}{\sqrt{n}} X_n - z)^{-1} - (-z)^{-1} = - \sum_{m=1}^\infty z^{-m-1} \left(\frac{1}{\sqrt{n}} X_n\right)^m$$
is absolutely convergent in the operator norm, uniformly in both $n$ and $z$, when $n$ is sufficiently large and $|z| \geq 1+2\eps$.  By the dominated convergence theorem, it thus suffices to show that the $k \times k$ matrix
$$ B_n (\frac{1}{\sqrt{n}} X_n)^m A_n$$
converges almost surely to zero for each fixed $m \geq 1$.  Breaking $B_n$ and $A_n$ into components, it suffices to show the following claim:

\begin{lemma}[Coefficient bound]\label{coef}  Let $X_n$ be an iid random matrix whose atom distribution has finite fourth moment.  Then
\begin{equation}\label{uvn}
 \langle (\frac{1}{\sqrt{n}} X_n)^m u_n, v_n \rangle = o(1)
\end{equation}
almost surely for each fixed $m \geq 1$ and any fixed (deterministic) sequence of unit vectors $u_n, v_n \in \C^n$.
\end{lemma}

We now prove the lemma.\footnote{We thank Jean Rochet for pointing out an error in a previous version of this argument.}

The atom distribution $x$ is currently only assumed to have finite fourth moment.  However, a standard truncation argument (using the results of \cite{BY} to control the contribution of the tail of $x$) shows that we may almost surely approximate $\frac{1}{\sqrt{n}} X_n$ to arbitrary accuracy in operator norm by an iid matrix in which the atom distribution is in fact bounded. As such, it will suffice to prove the lemma under the additional assumption that $x$ is bounded.  In particular, all moments of $x$ are now finite: $\E |x|^j = O(1)$ for all fixed $j$.

By diagonalising the covariance matrix of $\Re(x)$ and $\Im(x)$, we may assume (after a phase rotation) that $\Re(x)$ and $\Im(z)$ have zero covariance, and have variances $\sigma^2$ and $1-\sigma^2$ respectively for some $0 \leq \sigma^2 \leq 1$.  Next, by splitting $u_n, v_n$ into real and imaginary parts and renormalising, we may assume without loss of generality that $u_n, v_n$ have real coefficients.  Finally, by orthogonal decomposition of $v_n$, we may assume that $v_n$ is either equal to $u_n$ for all $n$, or is orthogonal to $u_n$ for all $n$.

We decompose
$$ \langle (\frac{1}{\sqrt{n}} X_n)^m u_n, v_n \rangle = 
\langle \pi((\frac{1}{\sqrt{n}} X_n)^m) u_n, v_n \rangle + \frac{1}{n} \tr (\frac{1}{\sqrt{n}} X_n)^m \langle u_n, v_n \rangle$$
where $\pi$ is the projection onto trace zero matrices.

From the circular law (Theorem \ref{thi}), using Theorem \ref{noout} to control the outliers, we have
$$ \frac{1}{n} \tr (\frac{1}{\sqrt{n}} X_n)^m = o(1)$$
almost surely.  Thus it suffices to show that
$$ \langle \pi((\frac{1}{\sqrt{n}} X_n)^m) u_n, v_n \rangle  = o(1)$$
almost surely.

We now use the moment method\footnote{We thank David Renfrew and Sean O'Rourke for pointing out an error in the initial version of this manuscript, which only established Lemma \ref{coef} in probability rather than in the almost sure sense.}.  It will suffice to show that
\begin{equation}\label{q}
\E |\langle \pi((\frac{1}{\sqrt{n}} X_n)^m) u_n, v_n \rangle|^4 = O(n^{-3/2}).
\end{equation}
Indeed, this implies from Markov's inequality that $\langle \pi((\frac{1}{\sqrt{n}} X_n)^m) u_n, v_n \rangle = O( n^{-1/16} )$ with probability $1-O(n^{-5/4})$, and Lemma \ref{coef} then follows from the usual truncation argument.  The bound $O(n^{-3/2})$ is not optimal (the truth should be $O(n^{-2})$), but will suffice for this purpose.

We first deal with the model case when $x$ is normally distributed with the the normal distribution $N(0,\sigma^2)_\R + i N(0,1-\sigma^2)_\R$ (with the real and imaginary parts independent).  As is well known, in this case the ensemble $X_n$ is invariant under conjugation by orthogonal matrices.  This implies that the expression $\E |\langle (\frac{1}{\sqrt{n}} X_n)^m u_n, v_n \rangle|^4$ does not change if we simultaneously conjugate $u_n, v_n$ by an orthogonal matrix.  In particular, if $u_n,v_n$ are orthogonal, we may keep $u_n$ deterministic, but replace $v_n$ by a random unit vector $w_n$ chosen uniformly from the orthogonal complement of $u_n$ on the unit sphere, independently of $X_n$.   However, a short computation in cylindrical coordinates (or Levy's concentration of measure theorem) shows that for any deterministic vector $x$, one has $\E |\langle x,w_n\rangle|^4 \ll \|x\|^4 / n^2$ if $w$ is a unit vector drawn uniformly from the orthogonal complement of $u_n$ on the sphere.   Thus we can bound the left-hand side of \eqref{q} by 
$$ O( \frac{1}{n^2} \E \|(\frac{1}{\sqrt{n}} X_n)^m u_n\|^4 ).$$
By Theorem \ref{noout}, the inner expectation is $O(1)$, and the claim \eqref{q} follows in the Gaussian case when $u_n$ and $v_n$ are orthogonal.

Now suppose that $u_n = v_n$.  By the conjugation invariance, it suffices to show that
$$
\E |\langle \pi((\frac{1}{\sqrt{n}} X_n)^m) y_n, y_n \rangle|^4 = O(n^{-3/2})$$
when $y_n$ is drawn uniformly from the unit sphere independently of $X_n$.  For fixed $X_n$, this expression depends on $y_n$ in a Lipschitz fashion with Lipschitz constant $O( \| \frac{1}{\sqrt{n}} X_n \|_{op}^{O(m)} )$, and its mean is zero since $\pi((\frac{1}{\sqrt{n}} X_n)^m)$ is trace zero.  The claim then follows from the Levy concentration of measure theorem and Theorem \ref{noout}.  This completes the proof of \eqref{q} in the Gaussian case.

To handle the non-gaussian case, it thus suffices to show that
$$ \E |\langle (\frac{1}{\sqrt{n}} X_n)^m u_n, v_n \rangle|^2 - \E |\langle (\frac{1}{\sqrt{n}} G_n)^m u_n, v_n \rangle|^2 = O(n^{-3/2}).$$
\begin{align*}
&n^{-2m} \sum_{a_0,\ldots,a_m,b_0,\ldots,b_m,c_0,\ldots,c_m,d_0,\ldots,d_m \in \{1,\ldots,n\}} \\
& \quad u_{n,a_0} \overline{v_{n,a_m}} \overline{u_{n,b_0}} v_{n,b_m} u_{n,c_0} \overline{v_{n,c_m}} \overline{v_{n,d_0}} v_{n,d_m} \\
& \quad (\E \prod_{j=1}^m x_{a_j,a_{j-1}} \overline{x_{b_j,b_{j-1}}} x_{c_j,c_{j-1}} \overline{x_{d_j,d_{j-1}}} \\
&\quad\quad - \E \prod_{j=1}^m g_{a_j,a_{j-1}} \overline{g_{b_j,b_{j-1}}} g_{c_j,c_{j-1}} \overline{g_{d_j,d_{j-1}}} )
\end{align*}
where $u_{n,i}, v_{n,i}$ are the coordinates of the unit vectors $u_n,v_n$, and $g_{ij}$ are iid copies of the complex normal distribution $N(0,\sigma^2)_\R + i N(0,1-\sigma^2)_\R$.

Consider the collection of $4m$ ordered pairs 
\begin{equation}\label{tree}
 (a_{j-1},a_j), (b_{j-1},b_j), (c_{j-1},c_j), (d_{j-1},d_j)
\end{equation}
for $j=1,\ldots,m$.  From the iid nature of $X$, and the fact that the random variables $x$ and $g$ match up to second order, we see that each summand in the above expression vanishes unless each ordered pair appears with multiplicity at least two, and at least one pair appears at least three times; in particular, there are at most $2m-1$ distinct ordered pairs.  As $x, g$ have all moments finite, we may thus bound the above expression in magnitude by
$$ O( n^{-2m} \sum_* |u_{n,a_0}| |v_{n,a_m}| |u_{n,b_0}| |v_{n,b_m}| |u_{n,c_0}| |v_{n,c_m}| |u_{n,d_0}| |v_{n,d_m}| )$$
where $\sum_*$ denotes the sum over all tuples $a_0,\ldots,a_m,b_0,\ldots,b_m,c_0,\ldots,c_m,d_0,\ldots,d_m$ such that the ordered pairs \eqref{tree} are such that each pair occurs at least twice, and at least one pair occurs three or more times; in particular
there are at most $2m-1$ distinct pairs.   Our task is now to show that
\begin{equation}\label{vsl}
\sum_* |u_{n,a_0}| |v_{n,a_m}| |u_{n,b_0}| |v_{n,b_m}| |u_{n,c_0}| |v_{n,c_m}| |u_{n,d_0}| |v_{n,d_m}| \ll n^{2m-3/2}
\end{equation}

Suppose that $a_0,\ldots,a_m,b_0,\ldots,b_m,c_0,\ldots,c_m,d_0,\ldots,d_m$ are such that \eqref{tree} is of the stated form.  Let $G$ be the unordered looped graph $G$ with edges being the unordered pairs associated to \eqref{tree} and with vertices being the elements of these pairs, and let $r$ be the number of connected components of $G$.  Then $1 \leq r \leq 4$, and $G$ has at most $2m-1$ edges and thus at most $2m+r-1$ vertices.

We consider first the contribution $\sum_{**}$ of those tuples for which $G$ has at most $2m+r-2$ vertices; this is for instance the case if $G$ contains a cycle or a looped edge, or has strictly fewer than $2m-1$ edges.  Then if one fixes $a_0,b_0,c_0,d_0$, then one has fixed at least one vertex in each component of $G$, leaving at most $2m-2$ remaining vertices.  Thus there are $O(2^{2m-2})$ choices for the remaining data $a_1,\ldots,a_m,b_1,\ldots,b_m,c_1,\ldots,c_m,d_1,\ldots,d_m$; summing over $a_0,b_0,c_0,d_0$ using the fact that $u_n$ is a unit vector yields that
$$ \sum_{**} |u_{n,a_0}|^2 |u_{n,b_0}|^2 |u_{n,c_0}|^2 |u_{n,d_0}|^2 \ll n^{2m-2}$$
and similarly
$$ \sum_{**} |v_{n,a_m}|^2 |v_{n,b_m}|^2 |v_{n,c_m}|^2 |v_{n,d_m}|^2 \ll n^{2m-2}$$
and so by Cauchy-Schwarz the contribution of the $\sum_{**}$ tuples to \eqref{vsl} is acceptable.

Now consider the contribution $\sum_{***}$ of those tuples for which $G$ has exactly $2m+r-1$ vertices, but such that two of the $a_0,b_0,c_0,d_0$ are distinct elements of a common component of $G$.  Then as in the $\sum_{**}$ case, fixing $a_0,b_0,c_0,d_0$ leaves only $O(2^{2m-2})$ choices for the remaining data, so that
$$ \sum_{***} |u_{n,a_0}|^2 |u_{n,b_0}|^2 |u_{n,c_0}|^2 |u_{n,d_0}|^2 \ll n^{2m-2}$$
while we also have the cruder bound
$$ \sum_{***} |v_{n,a_m}|^2 |v_{n,b_m}|^2 |v_{n,c_m}|^2 |v_{n,d_m}|^2 \ll n^{2m-1}$$
so by Cauchy-Schwarz this contribution is also acceptable.  Similarly if $G$ has $2m+r-1$ vertices but two of the $a_m,b_m,c_m,d_m$ are distinct elements of a common component of $G$.

The only remaining contribution $\sum_{****}$ comes from the case when $G$ has exactly $2m+r-1$ vertices, the $a_0,b_0,c_0,d_0$ agree whenever they lie in a common component of $G$, and $a_m,b_m,c_m,d_m$ agree whenever they lie in a common component of $G$.  As discussed previously, the requirement that $G$ has exactly $2m+r-1$ vertices forces $G$ to be a forest (a union of $r$ disjoint trees, with no cycles or looped edges) and to contain exactly $2m-1$ edges.  Among other things, this implies that the tuples \eqref{tree} do not contain a loop $(a,a)$, nor do these tuples contain a pair $(a,b)$ together with its reversal $(b,a)$.  If $a_0,\ldots,a_m$ and $b_0,\ldots,b_m$ (for instance) lie in the same component of $G$, then we must then $a_0=b_0$ and $a_m=b_m$, and then $a_i=b_i$ for all $0 \leq i \leq m$ (otherwise there would be a cycle, looped edge, or a pair $(a,b)$ and its reversal).  Thus each component has $m$ edges, which is only consistent with the total edge count of $2m-1$ if $r=1$ and $m=1$  But in this case the left-hand side of \eqref{vsl} simplifies to
$$ \sum_{i \neq j} |u_{n,i}|^4 |v_{n,j}|^4 $$
which is easily seen to be $O(1)$, so that \eqref{vsl} easily follows in this case. This completes the proof of Lemma \ref{coef} and hence Theorem \ref{small-rank}.

\section{Large mean}\label{nzm-sec}

Now we prove Theorem \ref{nonzero-mean}.  It suffices to show that for any fixed $\eps > 0$, almost surely one has exactly one eigenvalue of $\frac{1}{\sqrt{n}} X_n + \mu \sqrt{n} \phi_n \phi_n^*$ outside of the disk $\{ z: |z| \leq 1+2\eps\}$, with this eigenvalue occuring within $O(\eps)$ of $\mu \sqrt{n}$.

Fix $\eps$.  By Theorem \ref{noout}, almost surely there are no eigenvalues of $\frac{1}{\sqrt{n}} X_n$ outside of the disk. By Lemma \ref{eigencrit}, the eigenvalues of $\frac{1}{\sqrt{n}} X_n + \mu \sqrt{n} \phi_n \phi_n^*$ outside this disk are then precisely the solutions (counting multiplicity) to the equation $f(z) = 0$, where $f$ is the meromorphic function
$$ f(z) := 1 + \mu \sqrt{n} \left\langle (\frac{1}{\sqrt{n}} X_n - z)^{-1} \phi_n, \phi_n \right\rangle.$$
By Neumann series, we may expand
$$ f(z) = g(z) - \mu \sqrt{n} \sum_{m=1}^\infty z^{-m-1} \left\langle (\frac{1}{\sqrt{n}} X_n)^m \phi_n, \phi_n \right\rangle$$
where
\begin{equation}\label{g-def} 
g(z) := 1 - \mu \sqrt{n} / z.
\end{equation}
From Theorem \ref{noout}, we almost surely have
$$ \left\| (\frac{1}{\sqrt{n}} X_n) \right\|_{op} \leq 3$$
(say) and
$$ \left\| (\frac{1}{\sqrt{n}} X_n)^{m_0} \right\|_{op} \leq (1+\eps)^{m_0}$$
for some fixed integer $m_0$ depending only on $\eps$, and all sufficiently large $n$; this gives us the truncated Taylor expansion
$$ f(z) = g(z) - \mu \sqrt{n} \sum_{m=1}^M z^{-m-1} \left\langle (\frac{1}{\sqrt{n}} X_n)^m \phi_n, \phi_n \right\rangle
+ O\left( \frac{(1+\eps)^M}{|z|^{M+2}} \sqrt{n} \right)
$$
for any $M \geq 1$ and $|z| \geq 1+2\eps$, where the implied constant is allowed to depend on $\eps, \mu$ but not on $M$.  Applying Lemma \ref{coef}, we almost surely obtain
$$ f(z) = g(z) + o\left( \frac{\sqrt{n}}{|z|^2} \right) + O\left( \frac{(1+\eps)^M}{|z|^{M+2}} \sqrt{n} \right)
$$
for any fixed $M$ uniformly for all $|z| \geq 1+2\eps$, and thus (by letting $M$ go slowly to infinity) we almost surely have
$$ f(z) = g(z) + o\left( \frac{\sqrt{n}}{|z|^2} \right) $$
uniformly for all $|z| \geq 1+2\eps$.  From this and \eqref{g-def} we see that $f$ has no zeroes in this region except within a $o(1)$ neighbourhood of $\mu \sqrt{n}$, and from Rouche's theorem we see that there is exactly one zero of that latter type when $n$ is sufficiently large, and the claim follows.

\begin{remark}  The eigenvector corresponding to the exceptional zero can be explicitly described, by observing the identity
$$ (\frac{1}{\sqrt{n}} X_n + \mu \sqrt{n} \phi_n \phi_n^* - z) (1 - \frac{1}{z\sqrt{n}} X_n)^{-1} \phi_n = \frac{-f(z)}{z} \phi_n$$
for all non-zero $z$ outside of the spectrum of $\frac{1}{\sqrt{n}} X_n$.  In particular, if $z = \mu \sqrt{n}+o(1)$ is the outlier eigenvalue, the 
$(1 - \frac{1}{z\sqrt{n}} X_n)^{-1} \phi_n$ is the corresponding eigenvector.  From Theorem \ref{noout} and Neumann series, we see that almost surely, this eigenvector lies within $O(1/\sqrt{n})$ of $\phi_n$ in $\ell^2$ norm, and a more accurate description of this eigenvector can be given by expanding out the Neumann series further.
\end{remark}

\section{A central limit theorem}\label{clt-sec}

In Lemma \ref{coef}, we showed that the coefficients $ \langle (\frac{1}{\sqrt{n}} X_n)^m u_n, v_n \rangle$ decayed almost surely to zero.  Now we prove a more refined statement on the rate of decay, which is to Lemma \ref{coef} as the central limit theorem is to the (strong) law of large numbers.  This result will be needed to prove Theorem \ref{ull}.  For simplicity we consider only real-valued matrices; there is a complex analogue when the real and imaginary parts of $x$ have the same covariance matrix as the complex Gaussian $N(0,1)_\C$, but we will not state it here.

\begin{proposition}[Central limit theorem]\label{clt}  Let $X_n$ be an iid random matrix whose atom distribution is real-valued, has mean zeero and has all moments finite.  Let $u_n, v_n \in \R^n$ be a (deterministic) sequence of unit vectors whose coefficients $u_{n,i}, v_{n,i}$ are \emph{asymptotically delocalised} in the sense that
\begin{equation}\label{sloop}
 \sup_{1 \leq i \leq n} |u_{n,i}|, \sup_{1 \leq i \leq n} |v_{n,i}| = O(1/\sqrt{n}).
\end{equation}
Then for any fixed $m \geq 1$, the $m$ random variables
\begin{equation}\label{uvn-d}
Z_j := \sqrt{n} \left\langle (\frac{1}{\sqrt{n}} X_n)^j u_n, v_n \right\rangle
\end{equation}
for $j=1,\ldots,m$ converge jointly in distribution to the law of $m$ independent copies of the real gaussian $N(0,1)_\R$.
\end{proposition}

We now prove this proposition.  By (the multidimensional version of) Carleman's theorem (see e.g. \cite{bai}), it suffices to prove the moment bounds
\begin{equation}\label{zrr}
\E \prod_{j=1}^m Z_j^{r_j} = \E \prod_{j=1}^m G_j^{r_j} + o(1)
\end{equation}
for any natural numbers $r_1,\ldots,r_m$, where $G_1,\ldots,G_m$ are iid copies of the real gaussian $N(0,1)_\R$.

Fix $m,r_1,\ldots,r_m$; we allow all implied constants to depend on these quantities.  The left-hand side can be expanded as
\begin{equation}\label{summ}
 n^{-A} \sum_* \E \prod_{j=1}^m (\prod_{l=1}^{r_j} U_{a_{j,l,0}} V_{a_{j,l,j}}) (\prod_{l=1}^{r_j} \prod_{i=1}^j x_{a_{j,l,i},a_{j,l,i-1}}) 
\end{equation}
where $U_i := \sqrt{n} u_{n,i}$, $V_i := \sqrt{n} v_{n,i}$,
$$ A := \frac{1}{2} \sum_{j=1}^m r_j (j+1)$$
and $*$ ranges over all tuples of indices $a_{j,l,i} \in \{1,\ldots,n\}$ with $1 \leq j \leq m$, $1 \leq l \leq r_j$, $0 \leq i \leq j$.

From \eqref{sloop} and the bounded moment hypotheses, we see that each summand \eqref{summ} is of size $O(1)$, and so we may freely ignore up to $o(n^A)$ summands whenever desired.  

For each tuple in the sum $\sum_*$, we consider the $\sum_{j=1}^m j r_j$ ordered pairs $(a_{j,l,i-1},a_{j,l,i})$ with $1 \leq j \leq m$, $1 \leq l \leq r_j$, $1 \leq i \leq j$.  By the iid and mean zero nature of the $x_{i,j}$, we see that this sum vanishes unless each ordered pair appears with multiplicity at least two.  In particular each of the indices $a_{j,l,i}$ with $0 \leq i \leq j$ must occur with multiplicity at least $2$, leading to at most $A$ distinct indices.  If there are any fewer than $A$ distinct indices, then the contribution of this case to \eqref{summ} is $o(n^A)$, so we may assume that there are exactly $A$ distinct indices, which implies that each index $a_{j,l,i}$ occurs with multiplicity exactly two.  This implies that for $1 \leq i \leq j$, each of the $a_{j,l,i-1}$ arise exactly twice as the initial vertex of an ordered pair, and each of the $a_{j,l,i}$ arise exactly twice as the final vertex of an ordered pair.   Furthermore, the indices $a_{j,l,0}$ must be distinct from the indices $a_{j,l,i}$ with $0 < i \leq j$, and similarly the $a_{j,l,j}$ must be distinct from the indices $a_{j,l,i}$ with $i<k$, as otherwise the fact that each ordered pair appears at least twice will lead to a multiplicity of at least three at the repeated index.  This implies that the $\sum_j r_j$ paths $(a_{j,l,0},\ldots,a_{j,l,j})$ are simple paths, which each occur with multiplicity two but are otherwise disjoint.  The total contribution of this case to \eqref{summ} is then
$$ n^{-A} \sum_{**} \prod_{j=1}^m (\prod_{l=1}^{r_j} U_{a_{j,l,0}} V_{a_{j,l,j}}),$$
where the sum $\sum_{**}$ is over collections of simple paths $(a_{j,l,0},\ldots,a_{j,l,j})$ in $\{1,\ldots,n\}$ which occur with multiplicity two but are otherwise disjoint; here we use the fact that each ordered pair appears exactly twice, and that $x$ has unit variance.

In order for all paths to appear with multiplicity two, each of the $r_j$ must be even.  This already gives \eqref{zrr} when at least one of the $r_j$ is odd, so we now assume that the $r_j$ are all even. There are $\prod_{j=1}^m \frac{r_j!}{2^{r_j/2} (r_j/2)!}$ different ways in which the paths can be matched up to multiplicity two.  Once one fixes such a matching, there are $R := \frac{1}{2} \sum_{j=1} r_j$ initial vertices $b_1,\ldots,b_R$ of paths, and $R$ final vertices $c_1,\ldots,c_R$, all distinct from each other; if one fixes these vertices, then one has $(1+o(1)) n^{\sum_{j=1}^m (j-1) r_j/2} = (1+o(1)) n^{A - 2R}$ ways to choose the remaining paths.  This gives a total contribution to \eqref{summ} of
$$ (1+o(1)) n^{-2R} \prod_{j=1}^m \frac{r_j!}{2^{r_j/2} (r_j/2)!} \sum_{b_1,\ldots,b_R,c_1,\ldots,c_R \hbox{ distinct}}
\prod_{r=1}^R U_{b_r}^2 V_{c_r}^2.$$
We add back in the contributions in which some of the $b_1,\ldots,b_R,c_1,\ldots,c_R$; this only affects the sum by $o(n^{2R})$, which is acceptable.   We are left with
$$ (1+o(1)) n^{-2R} \prod_{j=1}^m \frac{r_j!}{2^{r_j/2} (r_j/2)!} \sum_{b_1,\ldots,b_R,c_1,\ldots,c_R \in \{1,\ldots,n\}}
\prod_{r=1}^R U_{b_r}^2 V_{c_r}^2;$$
since the $U_b$ and $V_c$ square-sum to $n$, this simplifies to
$$ (1+o(1)) \prod_{j=1}^m \frac{r_j!}{2^{r_j/2} (r_j/2)!} $$
and \eqref{zrr} follows from the standard computation
$$ \E G^r = \frac{r!}{2^{r/2} (r/2)!}$$
for $r$ even.

\section{Large non-selfadjoint perturbation}\label{ull-sec}

We now prove Theorem \ref{ull}.  It suffices to work in the exterior region $\Omega := \{ z: |z| \geq 1+4\eps \}$ for a fixed $\eps > 0$.  Henceforth all implied constants may depend on $\eps$, $p$, $\mu$.

\subsection{Crude upper bound}

We first show the first part of the theorem, namely that $\sup_n \E N_{4\eps}^m < \infty$ for all $m \geq 1$.  Fix $m$; we allow all implied constants to depend on $m$.  Our task is now to show that $\E N_{4\eps}^m = O(1)$ for all sufficiently large $n$.

From Theorem \ref{noout} we know that the spectral radius of $\frac{1}{\sqrt{n}} M_n$ is at most $1+\eps$ with overwhelming probability.  Using the trivial bound $N_{4\eps} \leq n$, we see that the tail event when the spectral radius exceeds $1+\eps$ thus gives a negligible contribution to $\E N_{4\eps}^m$ and will thus be ignored.

Conditioning on the event that the spectral radius is at most $1+\eps$, we may then apply Lemma \ref{eigencrit} to conclude that the eigenvalues of $A_n$ in $\Omega$ are precisely the zeroes in $\Omega$ of the random analytic function
\begin{equation}\label{fdef}
 f(z) := 1 + \mu \sqrt{n} \left\langle (\frac{1}{\sqrt{n}} X_n - z I)^{-1} \phi_n, \psi_n \right\rangle.
\end{equation}
In particular, $N_{4\eps}$ is the number of zeroes of $f$ in $\Omega$.

The function $f(1/z)$ is analytic in the disk $\{ w \in \C: |w| < \frac{1}{1+\eps} \}$ and equals $1$ at the origin, with $N_{4\eps}$ zeroes in the region $\{ w \in \C: |w| < \frac{1}{1+4\eps}\}$. Applying Jensen's formula to this function, we conclude the upper bound
$$ N_{4\eps} \ll \int_{|z| = r} \log_+ \frac{1}{|f(z)|}\ |dz|$$
for any radius $r$ between $1+2\eps$ and $1+3\eps$ (note that we allow implied constants to depend on $\eps$), where $\log_+ x := \max(\log x, 0)$ and $|dz|$ is arclength measure.  Averaging, we conclude that
\begin{equation}\label{neps}
 N_{4\eps} \ll \int_{1+2\eps \leq |z| \leq 1+3\eps} \log_+ \frac{1}{|f(z)|} \ d^2z. 
\end{equation}

It will thus suffice to establish the bound
\begin{equation}\label{loga}
 \E \int_K \log_+^m \frac{1}{|f(z)|} \ll 1
\end{equation}
for any fixed $m \geq 1$ and any compact subset $K$ of the annulus $|z| > 1+\eps$ (allowing implied constants to depend on $m,K$ of course).

We now pause to regularise the logarithm slightly, as this will come in handy later.  By Remark \ref{explicit}, the function $\log \frac{1}{|f(z)|}$ is a linear combination of $O(n)$ terms of the form $\log |z-z_0|$ for various complex numbers $z_0$.  As such we have a crude upper bound of the form
\begin{equation}\label{2e}
 \int_K \log_+^{m+1} \frac{1}{|f(z)|} \ d^2 z \ll n^{m+1}
\end{equation}
thanks to the triangle inequality.  From this, we have
$$ \int_K 1_{\log_+ \frac{1}{|f(z)|} \geq n^{m+1}} \log_+^{m} \frac{1}{|f(z)|} \ d^2 z \ll 1,$$
and so it suffices to show that
$$ \E \int_K \min( \log_+^{m} \frac{1}{|f(z)|}, n^{m+1} )\ d^2 z \ll 1.$$
By the Fubini-Tonelli theorem, it thus suffices to show that
$$ \E \min\left( \log_+ \frac{1}{|f(z)|}, n^{m+1} \right)^m = O(1)$$
uniformly for all $z$ in $K$.

It will then suffice to show the following lower tail estimates on $f$:

\begin{lemma}[Lower tail estimates]\label{ltail}  Let $z$ be in a compact subset $K$ of $\{ z \in \C: |z| > 1+\eps \}$; we allow implied constants to depend on $K$.
\begin{enumerate}
\item[(i)] For every $A > 0$ there exists $B > 0$ (not depending on $n$) such that
$$ \P( |f(z)| \leq n^{-B} ) \ll n^{-A}.$$
\item[(ii)] For every $\delta > 0$ one has
$$ \P( |f(z)| \leq \delta ) \ll \delta + n^{-c}$$
for some absolute constant $c>0$.
\end{enumerate}
\end{lemma}

Indeed, item (i) (with $A = 2m$) allows one to reduce to the case when $\log_+ \frac{1}{|f(z)|} = O(\log n)$ (it is here that we take advantage of our previous regularisation of the logarithm), and then (ii) and a dyadic decomposition gives the claim.

\begin{proof}
We begin with (i).  From \eqref{ann0} and \eqref{fdef} we have the identity
$$
(A_n - z I) (\frac{1}{\sqrt{n}} X_n - z I)^{-1} \phi_n
= f(z) \phi_n$$
and hence
$$ |f(z)| \geq \frac{\sigma_n(A_n - zI)}{\| \frac{1}{\sqrt{n}} X_n - z I \|_{op}},$$
where $\sigma_n(A_n-zI)$ is the least singular value of $A_n-zI$.  By Theorem \ref{noout}, we have $\| \frac{1}{\sqrt{n}} X_n - z I \|_{op}=O(1)$ with overwhelming probability.  The claim now follows from the least singular value bounds in \cite[Lemma 4.1]{tv-circular2}, \cite[Theorem 2.1]{TV-circular}, or \cite[Theorem 4.1]{lsv}, since we may express $A_n-zI$ as the sum of the normalised iid random matrix $\frac{1}{\sqrt{n}} X_n$ and the deterministic matrix $\mu \sqrt{n} \phi_n \psi_n^* - zI$, which has polynomial size.

Now we prove (ii).  Let $v$ be the vector $v := (\frac{1}{\sqrt{n}} X_n - z I)^{-1} \phi_n$.  It will suffice to show that
$$ \P(\sqrt{n} \langle v, \psi_n \rangle \in I) \ll |I| + n^{-c}$$
for every interval $I$.  From Theorem \ref{noout} and Neumann series, we see that with overwhelming probability, 
\begin{equation}\label{over}
 \| (\frac{1}{\sqrt{n}} X_n - z I)^{-1}\|_{op}, \| (\frac{1}{\sqrt{n}} X_n - z I) \|_{op} = O(1)
\end{equation}
and so 
$$ 1 \ll \|v\| \ll 1.$$
Let $\delta > 0$ be a small quantity (independent of $n$) to be chosen later.  
Call $v$ \emph{delocalised} if we have $|v_i| \geq \delta/\sqrt{n}$ for at least $\delta n$ of the indices $i=1,\ldots,n$.  As the coefficients of $\psi_n$ are iid (and are independent of $v$), we see from the Berry-Ess\'een theorem (see e.g. \cite[Chapter XVI]{feller}) that the contribution of the delocalised $v$ are acceptable.  It thus suffices to show that $v$ is delocalised with probability $1-O(n^{-c})$ for some $c>0$.

We will use an epsilon-net argument (cf. \cite{litvak}, \cite{rudelson}).  If $v$ is not delocalised, then $v$ lies within $O(\delta)$ in $\ell^2$ norm of a vector $w$ of comparable to $1$ that is \emph{sparse} in the sense that it is supported on at most $\delta n$ indices, simply by restricting $v$ to those indices $i$ for which $|v_i| \geq \delta/\sqrt{n}$, and truncating all other coefficients to zero.  From \eqref{over} and the triangle inequality, we thus have
$$ \| (\frac{1}{\sqrt{n}} X_n - z I) w - \phi_n \| \ll \delta.$$
The number of possible supports for $w$ is at most $\exp( O( \delta \log \frac{1}{\delta} \times n ) )$, thanks to Stirling's formula.  Once one fixes the supports, one can cover the range of $w$ by $O( (1/\delta) )^{\delta n}$ balls in $\ell^2$ norm of radius $\delta$.  Thus, by moving $w$ by at most $O(\delta)$ if necessary, we may assume that $w$ lies in a net $\Sigma$ of cardinality
\begin{equation}\label{sigma-card}
 |\Sigma| \ll \exp( O( \delta \log \frac{1}{\delta} \times n ) ).
\end{equation}

For each fixed $w$, the expression $\sqrt{n} \| (\frac{1}{\sqrt{n}} X_n - z I) w - \phi_n \|$ is a convex, $1$-Lipschitz function of $X_n$ as measured using the Frobenius norm $\|A\|_F := (\tr(A^* A))^{1/2}$.  Applying\footnote{For an extensive discussion of concentration inequalities, see \cite{ledoux}.  Note that many other atom distributions also enjoy concentration inequalities, such as those distributions with the log-Sobolev property; we will not attempt to aim for maximal generality here.} either the Talagrand concentration inequality (if $x$ is bounded) or the L\'evy concentration inequality (if $x$ is Gaussian), we conclude that
$$ \P\left( \left|\sqrt{n} \| (\frac{1}{\sqrt{n}} X_n - z I) w - \phi_n \| - M\right| \geq \lambda \right) \ll e^{-c\lambda^2}$$
for some absolute constant $c>0$, where $M$ is the median value of $\sqrt{n} \| (\frac{1}{\sqrt{n}} X_n - z I) w - \phi_n \|$.  To compute this median, we first compute the second moment
$$ \E \left(\sqrt{n} \| (\frac{1}{\sqrt{n}} X_n - z I) w - \phi_n \|\right)^2.$$
Expanding this out and using the fact that $X_n$ is an iid random matrix, this simplifies to
$$ n ( \| z w + \phi_n\|^2 + \|w\|^2 ).$$
In particular, this expression is comparable to $n$.  From the concentration inequality, we conclude that
$$ \P\left( \| (\frac{1}{\sqrt{n}} X_n - z I) w - \phi_n \| \ll \delta \right) \ll \exp(-cn)$$
if $\delta$ is small enough.  Summing up over all $w \in \Sigma$ using \eqref{sigma-card} and the union bound, we obtain the claim.
\end{proof}

This concludes the proof of part (i) of Theorem \ref{ull}.

\subsection{Correlation functions}

Now we prove part (ii) of Theorem \ref{ull}.  From part (i), we have the bound
\begin{equation}\label{oma}
 \int_{\Omega^k} \rho^{(k)}_n(z_1,\ldots,z_k) d^2 z_1 \ldots d^2 z_k = O(1)
\end{equation}
for each fixed $k$.  A similar (but simpler) argument also shows that
\begin{equation}\label{omb}
\int_{\Omega^k} \rho^{(k)}_\infty(z_1,\ldots,z_k) d^2 z_1 \ldots d^2 z_k = O(1)
\end{equation}
(the point being that the Gaussian random Laurent series $g(z)$ has a Gaussian distribution at each $z$ with an explicitly computable variance, so one can easily control the moments of $(\log^+ \frac{1}{|g(z)|})^m$).  As a consequence of these bounds, we can control perturbations to the test functions $F$ that are small in the uniform norm.

From Theorem \ref{noout} we know that the spectral radius of $\frac{1}{\sqrt{n}} M_n$ is at most $1+\eps$ with overwhelming probability.  The tail event when the spectral radius exceeds $1+\eps$ is thus negligible for the purposes of computing the asymptotics of the correlation functions and will thus be ignored.

Conditioning on the event that the spectral radius is at most $1+\eps$, we may then apply Lemma \ref{eigencrit} to conclude that the eigenvalues of $A_n$ in $\Omega$ are precisely the zeroes in $\Omega$ of the random analytic function
$$ f(z) := 1 + \mu \sqrt{n} \left\langle (\frac{1}{\sqrt{n}} X_n - z I)^{-1} \phi_n, \psi_n \right\rangle.$$
Thus, up to errors of $o(1)$, the correlation function $\rho^{(k)}_{A_n}$ is equal on $\Omega$ to the correlation function of the zeroes of $f$ (defined as in \eqref{ck}).  

As in previous sections, once the spectral radius is at most $1+\eps$, we can expand $f$ as a convergent Neumann series
$$ f(z) = 1 - \mu \sum_{j=0}^\infty \frac{f_{n,j}}{z^{j+1}}$$
where
$$ f_{n,j} := \sqrt{n} \left\langle (\frac{1}{\sqrt{n}} X_n)^j \phi_n, \psi_n \right\rangle.$$
To control this expression properly, we will need to work instead with the truncated Neumann series
$$ f(z) = 1 - \mu \sum_{j=0}^J \frac{f_{n,j}}{z^{j+1}} + \frac{\mu}{z^{J+1}} R_J(z)$$
for any $J \geq 1$, where the remainder $R_J$ is given by the formula
$$ R_J(z) := \sqrt{n} \langle (\frac{1}{\sqrt{n}} X_n)^{J+1} (\frac{1}{\sqrt{n}} X_n - z I)^{-1} \phi_n, \psi_n \rangle.$$

We now obtain a concentration bound on $R_J(z)$:

\begin{lemma}\label{rconj} Let $A,J \geq 1$.  If $n$ is sufficiently large (depending on $A,J,\eps,k$), then for each $z$ with $|z| \geq 1+2\eps$ and all $\lambda > 0$, one has
$$ \P( |R_J(z)| \geq \lambda J ) \ll e^{-c \lambda^2} + n^{-A}$$
for some $c>0$ depending only on $\eps,p,k$.
\end{lemma}

\begin{proof}  Let $m_0$ be an integer such that $(m_0+1) < (1+\eps)^{m_0}$.  By Theorem \ref{noout}, we see that with overwhelming probability, we have
$$ \| (\frac{1}{\sqrt{n}} X_n) \|_{op} \ll 1$$
and
$$ \| (\frac{1}{\sqrt{n}} X_n)^{m_0} \|_{op} \leq (1+\eps)^{m_0}$$
and hence by Neumann series
$$ \| (\frac{1}{\sqrt{n}} X_n - z I)^{-1} \|_{op} \ll 1$$
(recall that we allow implied constants to depend on $\eps$).  Henceforth we condition on the above event.  By another application of Theorem \ref{noout}, we see that with overwhelming probability, we have
$$  \| (\frac{1}{\sqrt{n}} X_n)^{J+1} \|_{op} \ll J$$
and hence
$$ \| (\frac{1}{\sqrt{n}} X_n)^{J+1} (\frac{1}{\sqrt{n}} X_n - z I)^{-1} \phi_n \| \ll J.$$
Note that the random vector $\psi_n$ is independent of $X_n$.  The claim then follows from the Azuma-Hoeffding inequality.
\end{proof}

Meanwhile, we have the following law for the $f_{n,j}$:

\begin{lemma}  Let $J \geq 1$. As $n \to \infty$, the random variables $f_{n,0},\ldots,f_{n,J}$ converge jointly in distribution to $J+1$ iid copies $g_0,\ldots,g_J$ of the real normal distribution $N(0,1)_\R$.
\end{lemma}

\begin{proof}  By the central limit theorem, $f_{n,0} = \sqrt{n} \langle \phi_n, \psi_n \rangle$ converges in distribution to $N(0,1)_\R$.  Now freeze $\psi_n$ and thus $f_{n,0}$.  By the law of large numbers, $\psi_n$ has a norm of $1+o(1)$ with probability $1-o(1)$.  Conditioning on this event, we see from Proposition \ref{clt} that $f_{n,1},\ldots,f_{n,J}$ converge in distribution to $g_1,\ldots,g_J$.  Integrating out the conditioning on $\psi_n$, we obtain the claim.
\end{proof}

In view of this proposition and the Skorokhod representation theorem, we may thus find iid copies $g_0,g_1,g_2,\ldots$ of the real normal distribution $N(0,1)_\R$ (depending on $n$) that are coupled to the $f_{n,j}$ in such a way that
\begin{equation}\label{fjj}
 \sup_{0 \leq j \leq J} |f_{n,j} - g_j| = o(1)
\end{equation}
uniformly with probability $1-o(1)$, for each $J \geq 1$.  We thus have
$$ f(z) = 1 - \mu \sum_{j=0}^J \frac{g_j}{z^{j+1}} + \frac{\mu}{z^{J+1}} R_J(z) + o(1)$$
uniformly with probability $1-o(1)$, for any fixed $J$.  

Next, we introduce the function
$$ g(z) = 1 - \mu \sum_{j=0}^\infty \frac{g_j}{z^{j+1}}.$$
The correlation functions $\rho^{(k)}_\infty$ are the correlation functions of the zeroes of $g$.  It thus suffices to show that the correlation functions of the zeroes of $f$ converge in the vague topology to the correlation functions of the zeroes of $g$.  

For any fixed $z$ in $\Omega$, the tail $\mu \sum_{j=J+1}^\infty \frac{g_j}{z^{j-J}}$ of $g(z)$ is Gaussian with mean zero and variance $O( |z|^{-2J-2} )$.  It thus obeys the same tail bound as Lemma \ref{rconj} (indeed it obeys slightly better bounds).  From this, Lemma \ref{rconj}, and \eqref{fjj}
we conclude that
$$ \P( |f(z)-g(z)| \geq \lambda J / |z|^{J+1} ) \ll e^{-c'\lambda^2} + o(1)$$
for any $\lambda \geq 1$ and $J \geq 1$, where $c'>0$ is an absolute constant and the decay rate $o(1)$ can depend on $z$.  Letting $J \to \infty$, we conclude that $f(z)-g(z)$ converges in probability to zero for any fixed $z \in \Omega$.

Given any smooth compactly supported function $F: \Omega \to \C$, define the random variables
$$ F(\Lambda_f) := \sum_{w \in \Lambda_f} F(w)$$
and
$$ F(\Lambda_g) := \sum_{w \in \Lambda_g} F(w),$$
where $\Lambda_f,\Lambda_g$ are the zeroes of $f,g$ respectively (counting multiplicity).  By the Stone-Weierstrass theorem (using \eqref{oma}, \eqref{omb} to control errors that are small in the uniform norm), it suffices to show that
$$ \E F_1(\Lambda_f) \ldots F_k(\Lambda_f) = \E F_1(\Lambda_g) \ldots F_k(\Lambda_g) + o(1)$$
for all smooth compactly supported $F_1,\ldots,F_k: \Omega \to \C$.  From part (i) of Theorem \ref{ull}, we know that $F_1(\Lambda_f) \ldots F_k(\Lambda_f)$ is uniformly integrable in $n$ (indeed, it has bounded $L^m$ norm for each $m$).  Thus it suffices to show that $F_1(\Lambda_f) \ldots F_k(\Lambda_f) - F_1(\Lambda_g) \ldots F_k(\Lambda_g)$ converges in probability to zero.  By another appeal to Theorem \ref{ull}(i), it suffices
to show that $F_j(\Lambda_f)-F_j(\Lambda_g)$ converges in probability to zero for each $j$.

Fix $j$, and write $F$ for $F_j$.  By Green's theorem, we can write
$$ F(\Lambda_f) = \frac{1}{2\pi} \int_\C (\Delta F(z)) \log |f(z)|\ d^2 z$$
where $\Delta := \frac{\partial^2}{\partial x^2} + \frac{\partial^2}{\partial y^2}$ is the usual Laplacian.  Similarly for $F(\Lambda_g)$.  Thus it suffices to show that
$$ \int_\C (\Delta F(z)) (\log |f(z)| - \log |g(z)|)\ d^2 z$$
converges in probability to zero.  

We already know that for fixed $z$, $f(z)-g(z)$ converges in probability to zero.  The function $\Delta F$ is bounded and compactly supported.  Also, by Lemma
To conclude the claim it then suffices by a truncation argument (cf. \cite[Lemma 3.1]{tv-circular2}) to obtain the uniform integrability bounds
$$ \E \int_{\operatorname{supp}(F)} |\log |f(z)||^2 + |\log |g(z)||^2\ d^2 z = O(1).$$
But the bound for $f$ follows from \eqref{loga}; the bound for $g$ can be deduced from $f$ by a Fatou lemma type argument. The proof of Theorem \ref{ull} is now complete.

\section{Zero row sum}\label{zrc}

We now prove Theorem \ref{xp}.  We begin by proving the spectral radius upper bound
$$ \rho( \frac{1}{\sqrt{n}} X_n P_n ) \leq 1 + o(1)$$
which holds almost surely.  It will suffice to show that almost surely one has
$$ \| ( \frac{1}{\sqrt{n}} X_n P_n )^m \|_{op} \leq O(m^{O(1)}) + o(1)$$
for each $m \geq 1$. Writing $P_n = 1 - (1-P_n)$, expanding, and applying Theorem \ref{noout} and the fact that the operator norm forms a Banach algebra, it then suffices to show that
$$ \| (1-P_n) (\frac{1}{\sqrt{n}} X_n)^j (1-P_n) \|_{op} = o(1)$$
 almost surely for each fixed $j$, as this handles all but the $O(m)$ terms in the expansion that involve at most one factor of $1-P_n$, each of which is $O(m^2)$ at worst by Theorem \ref{noout}.  But this bound follows from Lemma \ref{coef}.
 
The spectral radius lower bound will follow from the circular law claim.  Since almost sure convergence implies convergence in probability by the dominated convergence theorem, it will suffice to show that $\mu_{\frac{1}{\sqrt{n}} X_n P_n}$ and $\mu_{\frac{1}{\sqrt{n}} X_n}$ have the same almost sure limit.  Applying the replacement principle (\cite[Theorem 2.1]{tv-circular2}), it suffices to show that for almost every complex number $z$, one has
\begin{equation}\label{npz}
 \frac{1}{n} \log |\det( \frac{1}{\sqrt{n}} X_n P_n - z )| - \frac{1}{n} \log |\det( \frac{1}{\sqrt{n}} X_n - z )|
\end{equation}
converges almost surely to zero.

Fix $z$; we may take $z$ to be non-zero.   We allow implied constants in the $O()$ notation to depend on $z$.  We can rewrite \eqref{npz} as
\begin{equation}\label{nupp}
\frac{1}{2} \left( \int_0^\infty \log t\ d\nu'_n(t) - \int_0^\infty \log t\ d\nu_n(t)\right)
\end{equation}
where $\nu'_n, \nu_n$ are the ESDs of $( \frac{1}{\sqrt{n}} X_n P_n - z) ( \frac{1}{\sqrt{n}} X_n P_n - z)^*$ and $( \frac{1}{\sqrt{n}} X_n - z) ( \frac{1}{\sqrt{n}} X_n - z)^*$ respectively. 

The matrix $\frac{1}{\sqrt{n}} X_n P_n - z$ is a rank one perturbation of $\frac{1}{\sqrt{n}} X_n - z$, and so the singular values $\sigma_{i}( \frac{1}{\sqrt{n}} X_n P_n - z )$ of the former interlace that of the latter in the sense that
\begin{equation}\label{interlace}
 \sigma_{i-1}( \frac{1}{\sqrt{n}} X_n - z ) \geq \sigma_{i}( \frac{1}{\sqrt{n}} X_n P_n - z ) \geq \sigma_{i+1}( \frac{1}{\sqrt{n}} X_n P_n - z )
\end{equation}
whenever $i$ is such that the expressions are well-defined (i.e. $1 < i \leq n$ for the first inequality and $1 \leq i < n$ for the second).  This adequately controls all of the singular values of $\frac{1}{\sqrt{n}} X_n P_n - z$ except for the smallest and largest.  But from Theorem \ref{noout} we know that the largest singular value of both matrices are $O(1)$.  From \cite[Lemma 4.1]{tv-circular2} we almost surely also have a lower bound 
$$ 
\sigma_n( \frac{1}{\sqrt{n}} X_n - z ) \gg n^{-O(1)}$$
for all sufficiently large $n$.  So if we can also obtain the corresponding bound
\begin{equation}\label{xpz}
\sigma_n( \frac{1}{\sqrt{n}} X_n P_n - z ) \gg n^{-O(1)}
\end{equation}
almost surely for all sufficiently large $n$, then by the alternating series test\footnote{More precisely, the integrals $\int_0^\infty \log t\ d\nu_n(t)$ and $\int_0^\infty \log t\ d\nu'_n(t)$ are both averages of $n$ increasing quantities of size $O( 1 + |\log n^{-O(1)}|)$; by the interlacing property \eqref{interlace} (which bounds the even terms in the latter average by the odd terms in the former, and vice versa), the difference between these two averages can be rearranged as an average of two alternating series whose terms are increasing in magnitude.} we see that
$$
\eqref{nupp} = O\left( \frac{1 + | \log n^{-O(1)} |}{n} \right) = o(1)$$
as required.  So it will suffice to establish the least singular value bound \eqref{xpz}.  By the Borel-Cantelli lemma, it will suffice to show that
$$
\P( \sigma_n( \frac{1}{\sqrt{n}} X_n P_n - z ) \leq n^{-C} ) = O( n^{-2} ) $$
(say) for all sufficiently large $n$, and some absolute constant $C$.  Taking transposes, it suffices to show that
$$
\P( \sigma_n( \frac{1}{\sqrt{n}} P_n X_n - z ) \leq n^{-C} ) = O( n^{-2} ).$$

Let $C$ be chosen later.  In order for the above event to hold, there must exist a unit vector $v$ such that
$$ \| \frac{1}{\sqrt{n}} P_n X_n v - z v \| \leq n^{-C}.$$
We now work to eliminate the role of the projection $P_n$ by dropping a dimension.  Taking inner products with $\phi_n$, we see that
$$ z |\langle v, \phi_n \rangle| \leq n^{-C}.$$
Since $z$ is fixed and non-zero, we thus see that
$$ v - P_n v = O( n^{-C+O(1)} )$$
and thus
$$ \frac{1}{\sqrt{n}} P_n X_n P_n v - z P_n v = O(n^{-C+O(1)}).$$
If we let $v' := v-P_n v / \|v-P_n v\|$, we thus see that $v'$ is orthogonal to $\phi_n$ and
$$ \frac{1}{\sqrt{n}} P_n X_n v' - z v' = O(n^{-C+O(1)})$$
and thus
$$ \frac{1}{\sqrt{n}} X_n v' - z v' = \alpha \phi_n + O(n^{-C+O(1)})$$
for some complex number $\alpha$.  Writing
$$ v' := (v'_1,\ldots,v'_{n-1}, -v'_1-\ldots-v'_n)$$
we thus have
$$ \sum_{j=1}^{n-1} [\frac{1}{\sqrt{n}} (x_{ij} - x_{in}) - z (\delta_{ij} - \delta_{in})] v'_j = \frac{\alpha}{\sqrt{n}} + O(n^{-C+O(1)})$$
for all $i=1,\ldots,n$.  Subtracting off the $i=n$ equation to eliminate $\alpha$, we conclude that
$$ \sum_{j=1}^{n-1} [\frac{1}{\sqrt{n}} (x_{ij} - x_{in} - x_{nj} + x_{nn}) - z (\delta_{ij} + 1)] v'_j = O(n^{-C+O(1)})$$
for all $i=1,\ldots,n-1$.  Since $v'$ has unit norm, $(v'_1,\ldots,v'_{n-1})$ has norm between $1$ and $1/2\sqrt{n}$, and we conclude that
\begin{equation}\label{sxn}
 \sigma_{n-1}( \frac{1}{\sqrt{n}} X_{n-1} + D_{n-1} ) = O(n^{-C+O(1)} )
\end{equation}
where $X_{n-1}, D_{n-1}$ are the $n-1 \times n-1$ matrices
$$ X_{n-1} := (x_{ij})_{1 \leq i,j \leq n-1}$$
and
$$ D_{n-1} := ( \frac{1}{\sqrt{n}} (-x_{in}-x_{nj}+x_{nn}) - z \delta_{ij} + z )_{1 \leq i,j \leq n-1}.$$
If we condition $x_{in}, x_{jn}, x_{nn}$ to be fixed, then $D_{n-1}$ is deterministic, while $X_{n-1}$ remains an iid random matrix.  Applying \cite[Lemma 4.1]{tv-circular2}, \cite[Theorem 2.1]{TV-circular}, or \cite[Theorem 4.1]{lsv}, we see that the conditional probability of \eqref{sxn} is $O(n^{-2})$ if $C$ is large enough, and if the $x_{in}, x_{nj}, x_{nn}$ are bounded by (say) $n^{100}$ in magnitude.  Integrating out the conditioning (and using Chebyshev's inequality and the union bound to handle the rare event when one of the entries $x_{in}, x_{nj}, x_{nn}$ is larger than $n^{100}$ in magnitude) we obtain the claim.  This concludes the proof of Theorem \ref{xp}.

\end{document}